\documentclass[a4paper,oneside,11pt]{amsart}

\reversemarginpar

\usepackage[colorlinks,hyperindex,linkcolor=blue,urlcolor=black,pdftitle={On similarity to contractions of class $C_{\cdot0}$ with finite defects},pdfauthor={Maria F.Gamal'}]
{hyperref}

\usepackage{amsmath}
\usepackage{amsfonts}
\usepackage{amssymb}
\usepackage{amsthm}

\usepackage[english]{babel}
\usepackage{enumerate}

\usepackage{graphicx}
\usepackage{color}

\usepackage[abbrev]{amsrefs}

\numberwithin{equation}{section}

\theoremstyle{plain}
\newtheorem{theorem}{Theorem}[section]
\newtheorem{lemma}[theorem]{Lemma}

\newtheorem{corollary}[theorem]{Corollary}
\newtheorem{proposition}[theorem]{Proposition}

\theoremstyle{definition}
\newtheorem{remark}[theorem]{Remark}
\newtheorem{example}[theorem]{Example}

\begin{document}

\title[On similarity to contractions of class $C_{\cdot 0}$ with  finite defects]{On similarity to  contractions of class $C_{\cdot 0}$  with  finite defects}

\author{Maria F. Gamal'}
\address{
 St. Petersburg Branch\\ V. A. Steklov Institute 
of Mathematics\\
Russian Academy of Sciences\\ Fontanka 27, St. Petersburg\\ 
191023, Russia  
}
\email{gamal@pdmi.ras.ru}



\begin{abstract}A criterion on the similarity of a (bounded, linear) operator $T$ 
on a (complex, separable) Hilbert space $\mathcal H$  to a contraction of 
class $C_{\cdot 0}$ with finite unequal defects is given in terms of 
shift-type invariant subspaces of $T$.  Namely, $T$ 
is similar to such a contraction if and only if there exists
a finite collection  of 
(closed) invariant subspaces $\mathcal M$ of $T$ 
 such that the restriction $T|_{\mathcal M}$ of $T$ on $\mathcal M$ is 
similar to the simple unilateral shift and 
 the linear span of these subspaces $\mathcal M$ is $\mathcal H$. 
 A sufficient condition for the similarity of an absolutely 
continuous polynomially bounded operator $T$ to a contraction of class 
$C_{\cdot 0}$ with finite equal defects is given. Namely, $T$ is 
similar to such a contraction if the (spectral) multiplicity of 
$T$ is finite and $B(T)=\mathbb O$, where $B$ is a finite 
product of Blaschke products with simple zeros satisfying the Carleson 
interpolating condition (a Carleson--Newman product).


 \end{abstract}

\maketitle

\section{Introduction}

 Let $\mathcal H$ be a (complex, separable) Hilbert space, 
and let  $\mathcal L(\mathcal H)$ be the algebra of all (bounded, linear)  operators acting on  $\mathcal H$. 
A (closed) subspace  $\mathcal M$ of  $\mathcal H$ is called \emph{invariant} 
for an operator $T\in\mathcal L(\mathcal H)$, if $T\mathcal M\subset\mathcal M$. 
 Denote by  $\operatorname{Lat}T$ the collection of all invariant subspaces of $T$. 
It is well known and easy to see that $\operatorname{Lat}T$ is a complete lattice 
with the inclusion as a partial order.  

The \emph{multiplicity} $\mu_T$ of an operator  $T\in \mathcal L(\mathcal H)$ 
 is the minimum dimension of its reproducing subspaces: 
\begin{equation}\label{mu}  \mu_T=\min\{\dim E: E\subset \mathcal H, \ \ 
\vee_{n=0}^\infty T^n E=\mathcal H \}.\end{equation}

For Hilbert spaces  $\mathcal H$ and $\mathcal K$, let    $\mathcal L(\mathcal H, \mathcal K)$ denote the space of (bounded, linear) 
transformations acting from $\mathcal H$ to $\mathcal K$.
For $T\in\mathcal L(\mathcal H)$ and  $R\in\mathcal L(\mathcal K)$ 
set
\begin{equation*}
\mathcal I(T,R)=\{X\in\mathcal L(\mathcal H, \mathcal K)\ :\ XT=RX\}.
\end{equation*} 
Then $\mathcal I(T,R)$ is the set of all transformations $X$ which  \emph{intertwine} $T$ and $R$.
Let $X\in\mathcal I(T,R)$.  
If $X$ is unitary, then $T$ and $R$ 
are called  \emph{unitarily equivalent}, in notation: $T\cong R$. If $X$ is invertible, that is, 
$X^{-1}\in\mathcal L(\mathcal K, \mathcal H)$, 
then $T$ and $R$ are called \emph{similar}, in notation: $T\approx R$. 
If there exists  $X\in\mathcal I(T,R)$  such that  $\ker X=\{0\}$ and $\operatorname{clos}X\mathcal H=\mathcal K$, then
$T$ is called a  \emph{quasiaffine transform} of $R$, in notation: $T\prec R$. 
If $T\prec R$ and $R\prec T$, then $T$ and $R$ are called  \emph{quasisimilar}, in notation: $T\sim R$.

An  operator $T$ is called  \emph{polynomially bounded}, if 
there exists a constant $C>0$ such that $\|p(T)\|\leq C\|p\|_\infty$ for every (analytic) polynomial $p$,
 where $\|p\|_\infty=\max_{|z|\leq 1}|p(z)|$. 
 $T$ is an \emph{absolutely continuous (a.c.)}  polynomially bounded operator, 
if and only if the $H^\infty$-functional calculus is well defined for $T$, see \cite{mlak} or \cite{ker16}. 
(The definition of  an a.c.  polynomially bounded operator is not recalled here.) 
An a.c.  polynomially bounded operator $T$ is called an \emph{operator of class $C_0$}, 
if there exists $0\not\equiv\varphi\in H^\infty$ such that $\varphi(T)=\mathbb O$. 
For every operator $T$ of class $C_0$ there exists a \emph{minimal function}, that is, an inner function $\theta$
 such that $\theta(T)=\mathbb O$ and if $\varphi\in H^\infty$ is such that $\varphi(T)=\mathbb O$,
 then $\varphi\in\theta H^\infty$, see \cite{bercpr} and {\cite[Proposition III.4.4]{sznfbk}} or {\cite[Sec. II.1]{berc}}. 
 
 An  operator $T$  is called a  \emph{contraction}, if $\|T\|\leq 1$. 
Every contraction $T$  is polynomially bounded with $C=1$ 
 by the von Neumann inequality (see, for example, {\cite[Proposition I.8.3]{sznfbk}}). 
If an operator of class $C_0$ is a contraction, then it is called a \emph{contraction of class $C_0$}, 
see \cite{berc} and \cite{sznfbk}.

A contraction  $T\in\mathcal L(\mathcal H)$  is called  \emph{of class} $C_{\cdot 0}$, if 
$\|(T^*)^nx\|\to 0$ for every $x\in\mathcal H$. If $T$ is a contraction of class $C_0$, 
then $T$ and $T^*$ are of class  $C_{\cdot 0}$ ({\cite[Proposition III.4.2]{sznfbk}}). 
The \emph{defect indices} (the \emph{defects} for brevity) of $T$ are $d_T=\dim(I-T^*T)\mathcal H$ and $d_{T^*}=\dim(I-TT^*)\mathcal H$, 
see  {\cite[Sec. I.3.1]{sznfbk}}.  If $T$ is a contraction of class $C_{\cdot 0}$, then 
$d_T\leq d_{T^*}$. Furthermore,  $T$ is unitarily equivalent to a compression 
on the orthogonal complement of an invariant subspace of 
the unilateral shift of multiplicity $d_{T^*}$. Consequently, $\mu_T\leq d_{T^*}$
({\cite[Corollary X.3.3]{sznfbk}}).  If $T$ is a contraction of class $C_{\cdot 0}$ and $N:=d_{T^*}<\infty$, 
then $T$ is of class $C_0$ if and only if $d_T=d_{T^*}$, see \cite{sznf74}. Contractions  $T$ of class $C_0$ with 
$d_{T^*}=N<\infty$ are called \emph{of class} $C_0(N)$. 
If $T$ is a contraction of class $C_{\cdot 0}$ and $d_T<d_{T^*}$, then $T$ has shift-type invariant subspaces, 
that is, there exist $\mathcal M\in\operatorname{Lat}T$ such that $T|_{\mathcal M}\approx S$, 
and the closed linear span of  such subspaces  is  the space $\mathcal H$  on which $T$ acts, see \cite{sznf74} and \cite{kerchy} 
or   {\cite[Sec. IX.3]{sznfbk}}.  
In \cite{gamal22}, it is proved that the minimal quantity of shift-type invariant subspaces of $T$ 
whose \emph{closed linear} span is  $\mathcal H$ is equal to 
$\mu_T$, if $2\leq\mu_T$, and is equal to $2$, if $\mu_T=1$. In this paper, it is proved that an operator $T$ 
is similar to a contraction $R$ of class  $C_{\cdot 0}$ with $d_R<d_{R^*}<\infty$  if and only if 
 the minimal quantity of shift-type invariant subspaces of $T$ 
whose \emph{linear} span is  $\mathcal H$ is finite. For  operators of class $C_0$ 
analogous results (with replacing shift-type invariant subspace by appropriate analogous subspaces) cannot hold true in general,  because operators of class $C_0$  can be unicellular.
 Remember that an operator $T$ is called \emph{unicellular}, if $\operatorname{Lat}T$ is totally ordered by inclusion. 
On the other side, if $B(T)=\mathbb O$ for a Carleson--Newman product $B$ and $\mu_T<\infty$, then 
$T$ is similar to a contraction of class $C_0(N)$ (with $N\in\mathbb N$). Similarity to  contractions of class $C_0(1)$, that is, Jordan blocks, is considered in \cite{clouatre11},  \cite{clouatre14},  \cite{clouatre15}. 

The paper is organized as follows. In Sec. 2 the criterion on similarity to contractions  $R$ of class  $C_{\cdot 0}$ with $d_R<d_{R^*}<\infty$ 
is proved. In Sec. 3 an example of a unicellular contraction of class $C_0(2)$ is given which is not similar to a contraction of class $C_0(1)$. 
In Sec. 4 operators $T$ of class $C_0$ such that $B(T)=\mathbb O$ for  a Carleson--Newman product $B$ are considered. 

 The following notation will be used. For a (closed) subspace $\mathcal M$ of a Hilbert space $\mathcal H$, by 
$P_{\mathcal M}$ and $I_{\mathcal M}$ the orthogonal projection from $\mathcal H$ onto $\mathcal M$ and 
 the identity operator on $\mathcal M$ are denoted,  respectively. 
By $\mathbb O$ the zero transformation acting between (maybe nonzero) spaces is denoted.

Symbols $\mathbb D$ and $\mathbb T$ denote the open unit disc
and the unit circle, respectively. The normalized Lebesgue measure on $\mathbb T$ is denoted by $m$. 
$H^2$ is  the Hardy--Hilbert space on $\mathbb T$. $H^\infty$ is the  
 Banach algebra of all bounded analytic functions  in $\mathbb D$; 
 $\|\varphi\|_\infty=\sup_{|z|< 1}|\varphi(z)|$ for $\varphi\in H^\infty$. 
The simple unilateral shift  $S$ is the operator  of multiplication by the independent variable    on $H^2$. 
For $N\in\mathbb N$ denote by  $H^2_N$ and by $S_N$   the orthogonal sum of $N$ copies of $H^2$ and $S_N$,  respectively. 
Then $S_N$ is the unilateral shift of multiplicity $N$ acting on $H^2_N$: $S_N\in\mathcal L(H^2_N)$ and $\mu_{S_N}=N$.  
For $N\in\mathbb N$, vectors from $H^2_N$ are columns of the height $N$ of functions from $H^2$. 

Let $M,N\in\mathbb N$. 
The space $H^\infty(\mathbb C^M,\mathbb C^N)$ is the space of bounded analytic functions on $\mathbb D$ 
whose values are (bounded, linear) operators acting from $\mathbb  C^M$ to $\mathbb C^N$. 
Let $\Phi\in H^\infty(\mathbb C^M,\mathbb C^N)$. Then $\Phi$ has nontangential boundary values 
$\Phi(\zeta)$ for a.e. $\zeta\in\mathbb T$ with respect to  $m$. 
Furthermore, $\Phi$ has 
a representation $\Phi=[\varphi_{nk}]_{n=1,\ldots, N\atop k=1, \ldots,M}$, where $\varphi_{nk}\in H^\infty$,
with respect to some  orthonormal bases of $\mathbb C^M$ and $\mathbb C^N$. The operator of multiplication by $\Phi$ 
acts from $H^2_M$ to $H^2_N$. For $\Phi\in H^\infty(\mathbb C^M,\mathbb C^N)$ set 
$\widetilde\Phi(z)=\Phi^*(\overline z)$, $z\in\mathbb D$. 
The function  $\Theta\in H^\infty(\mathbb C^M,\mathbb C^N)$ is called inner, if 
the multiplication by $\Theta$ is an isometry from $H^2_M$ to $H^2_N$. The function $\Theta$ is called $*$-inner, 
if $\widetilde\Theta$ is inner, and is called  inner from both sides, if $\Theta$ is inner and $*$-inner. The definition of 
outer and $*$-outer function is not recalled here. For references, see {\cite[Sec. V.2.3]{sznfbk}}

Let $\Theta\in H^\infty(\mathbb C^M,\mathbb C^N)$. 
Then $\Theta$ is inner if and only if $\Theta^*(\zeta)\Theta(\zeta)=I_{\mathbb C^M}$ for a.e. $\zeta\in\mathbb T$ 
with respect to  $m$, and then $M\leq N$ ({\cite[Proposition V.2.2]{sznfbk}}). For inner function 
 $\Theta\in H^\infty(\mathbb C^M,\mathbb C^N)$ set 
\begin{equation*} \mathcal H(\Theta)=H^2_N\ominus\Theta H^2_M \ \ \text{ and }\ \ 
  T_\Theta = P_{\mathcal H(\Theta)} S_N|_{\mathcal H(\Theta)}.
\end{equation*}
Then $T_\Theta$ is a contraction 
of class $C_{\cdot 0}$ with $d_{T^*}\leq N$. Conversely, for every contraction $T$ of class $C_{\cdot 0}$ with $d_{T^*}<\infty$ 
there exists an inner function $\Theta\in H^\infty(\mathbb C^{d_T},\mathbb C^{d_{T^*}})$ such that $T\cong T_\Theta$. 
(Such $\Theta$ is called the characteristic function of  a contraction $T$.) For references, see  {\cite[Ch. V, VI]{sznfbk}}, or 
{\cite[Sec. V.1]{berc}}, or {\cite[Ch. C.1]{nik02}}.

\section{On similarity to contractions of class $C_{\cdot 0}$ with unequal finite defects}

In this section, a criterion on the similarity of an  operator $T$  to a contraction of 
class $C_{\cdot 0}$ with finite unequal defects (Theorem \ref{thmmain}) will be proved.

\begin{lemma}\label{lemtau} Let $N\in\mathbb N$, $N\geq 2$, and let $\tau_{0n}\subset\mathbb T$, $n=1,\ldots, N$, be  
measurable sets such that $m(\tau_{0n})>0$ for every $n$, $1\leq n \leq N$. Then there exist measurable sets 
$\sigma_n$, $n=1,\ldots, N$, such that $\sigma_n\subset\tau_{0n}$, $m(\sigma_n)>0$, $\sigma_n\cap\sigma_k=\emptyset$, if $n\neq k$, 
$n,k=1, \ldots, N$, and  $m(\cup_{n=1}^N\sigma_n)=m(\cup_{n=1}^N\tau_{0n})$. 
\end{lemma}

\begin{proof} The lemma will be proved using induction. Let $N=2$. Without loss of generality,
 we may assume that $m(\tau_{01})\leq m(\tau_{02})$. Set $\tau_2=\tau_{02}\setminus\tau_{01}$. If $m(\tau_2)>0$, then 
set $\sigma_1=\tau_{01}$ and $\sigma_2=\tau_2$. If  $m(\tau_2)=0$, then 
$m(\tau_{01}\cup\tau_{02})=m(\tau_{01}\cap\tau_{02})$. Every representation $\tau_{01}\cap\tau_{02}=\sigma_1\cup\sigma_2$ with 
$\sigma_1$ and $\sigma_2$ such that $m(\sigma_k)>0$, $k=1,2$, and $\sigma_1\cap\sigma_2=\emptyset$
satisfies the conclusion of the lemma in this case.

If $N>2$, assume that the lemma is proved for all $M$, $2\leq M\leq N-1$. 
Without loss of generality,  we may assume that $m(\tau_{0N})=\min_{1\leq n\leq N} m(\tau_{0n})$. 
Set $\tau_n=\tau_{0n}\setminus\tau_{0N}$, $n=1,\ldots, N-1$. Without loss of generality,  we may assume that 
there exists $1\leq M\leq N$  such that $m(\tau_n)>0$, $1\leq n\leq M-1$, and $m(\tau_n)=0$, $M\leq n\leq N-1$. 
Thus, $m(\cup_{n=1}^N\tau_{0n})=m(\cup_{n=1}^{M-1}\tau_n\cup\tau_{0N})$ and $(\cup_{n=1}^{M-1}\tau_n)\cap\tau_{0N}=\emptyset$. 
By the inductive hypothesis, there exist  measurable sets 
$\sigma_n$, $n=1,\ldots, M-1$, such that $\sigma_n\subset\tau_n$, $m(\sigma_n)>0$, $\sigma_n\cap\sigma_k=\emptyset$, if $n\neq k$, 
$n,k=1, \ldots, M-1$, and  $m(\cup_{n=1}^{M-1}\sigma_n)=m(\cup_{n=1}^{M-1}\tau_n)$. 
Furthermore, $\tau_{0N}=\cup_{n=M}^{N-1}(\tau_{0N}\setminus\tau_{0n})\cup\cap_{n=M}^N\tau_{0n}$. The relations 
$m(\tau_{0N})\leq m(\tau_{0n})$ and $m(\tau_n)=0$, $M\leq n\leq N-1$, imply $m(\tau_{0N}\setminus\tau_{0n})=0$. Consequently, 
$m(\tau_{0N})=m(\cap_{n=M}^N\tau_{0n})$. Take a representation $\cap_{n=M}^N\tau_{0n}=\cup_{n=M}^N \sigma_n$ with 
$\sigma_n$ such that $m(\sigma_n)>0$, and $\sigma_n\cap\sigma_k=\emptyset$, if $n\neq k$, $n,k=M,\ldots,N$. Then $\sigma_n$, 
$n=1, \ldots, N$, satisfy the conclusion of the lemma.
\end{proof}

  \begin{lemma} \label{lem1} Let $N\in\mathbb N$, $N\geq 2$, 
 and let $c>0$. Let $\varphi_n$, $n=1,\ldots, N$, be such that $0\not\equiv\varphi_n\in H^\infty$,  
$\|\varphi_n\|_\infty\leq 1$
 for every $n$, $1\leq n \leq N$, and 
\begin{equation*} \mathop{\mathrm{ess \, inf}}_{\zeta\in\mathbb T}\sum_{n=1}^N|\varphi_n(\zeta)|^2=c^2>0.
\end{equation*}
Then there exist $\Psi\in H^\infty(\mathbb C^N, \mathbb C^N)$ and inner functions $\vartheta_n$, $n=1,\ldots, N$, 
such that $\Psi$ is invertible in $H^\infty(\mathbb C^N, \mathbb C^N)$ and 
\begin{equation*} \Psi[\varphi_n]_{n=1}^N=[\vartheta_n]_{n=1}^N.
\end{equation*}
\end{lemma}

\begin{proof}  Take $\delta_n$, $n=1,\ldots, N$, such that 
 $0<\delta_n<\|\varphi_n\|_\infty$ for all $n=1,\ldots, N$, and  $\sum_{n=1}^N\delta_n^2<c^2$. Set 
\begin{equation*}\tau_{0n}=\{\zeta\in\mathbb T:\ |\varphi_n(\zeta)|\geq \delta_n\}, \ \ n=1,\ldots, N.
\end{equation*}  
Then $m(\tau_{0n})>0$ for every $n=1,\ldots, N$ and $m(\cup_{n=1}^N\tau_{0n})=1$. By Lemma \ref{lemtau},  there exist measurable sets 
$\sigma_n$, $n=1,\ldots, N$, such that $\sigma_n\subset\tau_{0n}$, $m(\sigma_n)>0$, $\sigma_n\cap\sigma_k=\emptyset$, if $n\neq k$, 
$n,k=1, \ldots, N$, and  $m(\cup_{n=1}^N\sigma_n)=1$. 

Set $\delta=\min_{1\leq n\leq N}\delta_n$. Take $a$, $b>0$ such that $a^{N-1}>N!b$ and $\delta>a+(N-2)b$.
 Define outer functions $\kappa_{nk}$, $n,k=1, \ldots, N$, by the formulas
\begin{equation*} |\kappa_{nn}(\zeta)|=
\begin{cases} a, \text{ if }\zeta\in\mathbb T\setminus\sigma_n,\\ 1, \text{ if }\zeta\in\sigma_n,\end{cases} \ \  
|\kappa_{nk}(\zeta)|=
\begin{cases} b, \text{ if }\zeta\in\mathbb T\setminus\sigma_k,\\ 1, \text{ if }\zeta\in\sigma_k,\end{cases} \ \ n\neq k.
\end{equation*} 
Set $\kappa=\det[\kappa_{nk}]_{n,k=1}^N$. Then 
\begin{equation}\label{kappa}
|\kappa(z)|\geq a^{N-1}-N!b, \ \ z\in\mathbb D,
\end{equation} and 
\begin{equation}\label{kappaphi} \Bigl|\sum_{k=1}^N\kappa_{nk}(\zeta)\varphi_k(\zeta)\Bigr|\geq\delta-a-(N-2)b \  
\text{ for a.e. }\zeta\in\mathbb T,\  n=1,\ldots,N.
\end{equation} 
Indeed, 
\begin{equation*}
\kappa=\prod_{n=1}^N\kappa_{nn}+\sum_{\{n_j\}_{j=1}^N}(-1)^s\prod_{j=1}^N\kappa_{n_j j}, 
\end{equation*}
where $s=0,1$ (depending on  $\{n_j\}_{j=1}^N$), and $\{n_j\}_{j=1}^N$ is a permutation of $\{1,\ldots, N\}$ such that $n_j\neq j$ for at least one $j$.  (The sum is taken by all such permutations.)
We have
$|\prod_{n=1}^N\kappa_{nn}|=a^{N-1}$ a.e. on $\mathbb T$.   
Since $\kappa_{nn}$ are outer, we have $\prod_{n=1}^N\kappa_{nn}$ is outer. Thus, $\prod_{n=1}^N\kappa_{nn}(z)=a^{N-1}$ 
for every $z\in\mathbb D$. Let $\{n_j\}_{j=1}^N$ be a permutation of $\{1,\ldots, N\}$  such that $n_j\neq j$ for at least one $j$, and let $1\leq k\leq N$. 
Then there exists $l\neq k$, $1\leq l\leq N$,  such that $n_l\neq l$. Consequently, 
$|\kappa_{n_l l}(\zeta)|=b$ for a.e. $\zeta\in\sigma_k$, and $|\prod_{j=1}^N\kappa_{n_j j}(\zeta)|\leq b$ 
for  a.e. $\zeta\in\sigma_k$. Since  $m(\cup_{k=1}^N\sigma_k)=1$, we conclude that 
$|\prod_{j=1}^N\kappa_{n_j j}(\zeta)|\leq b$ 
for  a.e. $\zeta\in\mathbb T$. Consequently, $|\prod_{j=1}^N\kappa_{n_j j}(z)|\leq b$, $z\in\mathbb D$.
 Thus, 
\begin{equation*}
|\kappa(z)|\geq\Bigl|\prod_{n=1}^N\kappa_{nn}(z)\Bigr|-\sum_{\{n_j\}_{j=1}^N}
\Bigl|\prod_{j=1}^N\kappa_{n_j j}(z)\Bigr|\geq
 a^{N-1}-N!b, \ \ z\in\mathbb D.
\end{equation*}
The estimate \eqref{kappa} is proved.
Let $1\leq n\leq N$. If $\zeta\in\sigma_n$, then 
 \begin{equation*}\Bigl|\sum_{k=1}^N\kappa_{nk}(\zeta)\varphi_k(\zeta)\Bigr|\geq
|\kappa_{nn}(\zeta)\varphi_n(\zeta)|-\sum_{k\neq n}|\kappa_{nk}(\zeta)\varphi_k(\zeta)|\geq\delta_n-(N-1)b.
\end{equation*} 
If $\zeta\in\sigma_k$ for $k\neq n$,  then 
 \begin{equation*} \Bigl|\sum_{l=1}^N\kappa_{nl}(\zeta)\varphi_l(\zeta)\Bigr|\geq
|\kappa_{nk}(\zeta)\varphi_k(\zeta)|-\sum_{l\neq k}|\kappa_{nl}(\zeta)\varphi_l(\zeta)|\geq\delta_k-a-(N-2)b.
\end{equation*} 
The estimate \eqref{kappaphi} is proved. Set $\sum_{k=1}^N\kappa_{nk}\varphi_k=\vartheta_n\eta_n$, $n=1,\ldots, N$, 
where $\vartheta_n$ is inner and $\eta_n$ is outer. By \eqref{kappaphi}, $1/\eta_n\in H^\infty$, $n=1,\ldots, N$. 
Set $\Psi=[\kappa_{nk}/\eta_n]_{n,k=1}^N$. Then   $\det\Psi=\kappa/\prod_{n=1}^N\eta_n$. Consequently, $\Psi$ 
is invertible 
in $H^\infty(\mathbb C^N, \mathbb C^N)$. The equality in the conclusion of the lemma follows from the defiition of $\Psi$.
\end{proof}

\begin{corollary} \label{cor1} Lemma \ref{lem1} holds true without assumption $\varphi_n\not\equiv 0$ for every $n=1,\ldots, N$. 
\end{corollary}

\begin{proof} Denote by $M$ the quantity of $n$ for which $\varphi_n\not\equiv 0$. Then $1\leq M\leq N$. Multiplying the column  
$[\varphi_n]_{n=1}^N$ by the (invertible) $N\times N$-matrix $A_1$ whose columns is an appropriate 
permutation of the columns of the unit $N\times N$-matrix $I_N$, we obtain the column  
$[\varphi_{1n}]_{n=1}^N$ such that  $\varphi_{1n}\not\equiv 0$ for $n=1,\ldots, M$ and  $\varphi_{1n}\equiv 0$ for $n=M+1,\ldots, N$. 
Note that $\{\varphi_{1n}\}_{n=1}^N$ is a permutation of  $\{\varphi_n\}_{n=1}^N$. 
Let $A_2$ be the $N\times N$-matrix  whose elements of first columns with indices  $n=M+1,\ldots, N$ are equal to $1$,
 and all other elements of $A_2$ are equal to $0$. Set $[\varphi_{2n}]_{n=1}^N=(I_N+A_2)[\varphi_{1n}]_{n=1}^N$. 
Then  $\varphi_{2n}=\varphi_{1n}$ for  $n=1,\ldots, M$ and  $\varphi_{2n}=\varphi_{11}$ for $n=M+1,\ldots, N$. 
The column $[\varphi_{2n}]_{n=1}^N$ satisfy the assumption of Lemma \ref{lem1}. The matrix $\Psi(I_N+A_2)A_1$ is the required matrix.
\end{proof}

\begin{theorem}\label{thm1} Let $M,N\in\mathbb N$, $N\geq 2$,  $1\leq M< N$, 
and let $\Theta\in H^\infty(\mathbb C^{N-M}, \mathbb C^N)$ be an inner function. 
Then there exist $\mathcal M_n\in\operatorname{Lat}T_\Theta$, $n=1,\ldots, N$, such that $T_\Theta|_{\mathcal M_n}\approx S$ for every $n$, $1\leq n\leq N$, and 
$\mathcal H(\Theta)=\mathcal M_1+\ldots+\mathcal M_N$  in the following sense: for every $x\in\mathcal H$ there exist 
  $x_n\in\mathcal M_n$, $n=1,\ldots, N$, such that $x=x_1+\ldots+x_N$.
\end{theorem}

\begin{proof} Let $\Theta=\Theta_1\Omega$ be the $*$-canonical factorization of $\Theta$, see {\cite[(V.4.20)]{sznfbk}}. 
Then $\Theta_1\in  H^\infty(\mathbb C^{N-M}, \mathbb C^N)$ is inner $*$-outer 
and $\Omega\in H^\infty(\mathbb C^{N-M}, \mathbb C^{N-M})$ is inner from both sides. By \cite{sznf74} and \cite{tak}, there exists an 
 $*$-inner 
function $\Phi\in H^\infty(\mathbb C^N, \mathbb C^M)$ such that $\Phi\Theta_1=\mathbb O$. Let
\begin{equation*}\Phi=\left[\begin{matrix} \varphi_{11} & \ldots & \varphi_{1N}\\ 
 \ldots & \ldots & \ldots \\ \varphi_{M1} & \ldots & \varphi_{MN}\end{matrix}\right], 
\end{equation*}
where $\varphi_{kn}\in H^\infty$, $k=1,\ldots, M$, $n=1,\ldots, N$. Since $\Phi$ is $*$-inner, we have 
$\sum_{n=1}^N|\varphi_{kn}(\zeta)|^2=1$ for a.e. $\zeta\in\mathbb T$ and every $k=1,\ldots, M$. 
Let
 $\Psi\in H^\infty(\mathbb C^N, \mathbb C^N)$ be a function from Corollary \ref{cor1} applied to the 
functions $\varphi_{1n}$,  $n=1,\ldots, N$. Then $\Psi$ is invertible in $H^\infty(\mathbb C^N, \mathbb C^N)$ and 
\begin{equation*} \Psi[\varphi_{1n}]_{n=1}^N=[\vartheta_n]_{n=1}^N
\end{equation*}
for some inner functions $\vartheta_n$, $n=1,\ldots, N$.
Denote by $\Psi^\top$ the transpose of $\Psi$. Then 
\begin{equation*}
\Phi\Psi^\top=\left[\begin{matrix} \vartheta_1 & \ldots & \vartheta_N\\ 
  \psi_{21} & \ldots & \psi_{2N}\\ 
 \ldots & \ldots & \ldots \\ \psi_{M1} & \ldots & \psi_{MN}\end{matrix}\right]
\end{equation*}
for some functions $\psi_{kn}\in H^\infty$, $k=2,\ldots, M$, $n=1,\ldots, N$.

Let  $n$, $1\leq n\leq N$, be fixed. For $h\in H^2$ denote by $h_{(n,N)}$ the column of the height $N$ 
whose element on $n$th place is equal to $h$ 
and all other elements are equal to $0$.  Then  $h_{(n,N)}\in H^2_N$, and 
\begin{equation}\label{vartheta}
\Phi\Psi^\top h_{(n,N)}=\left[\begin{matrix} \vartheta_n h\\ 
  \psi_{2n}h\\ 
 \ldots  \\ \psi_{Mn}h \end{matrix}\right].
\end{equation}
Define $Y_n\in\mathcal L( H^2, \mathcal H(\Theta))$ by the formula
  $Y_n h=P_{\mathcal H(\Theta)}\Psi^\top h_{(n,N)}$, $h\in H^2$. 
Then $Y_n S=T_\Theta Y_n$ and $\|Y_n h\|\geq\|h\|$ for every $h\in H^2$.
Indeed, 
\begin{align*} Y_n Sh&=P_{\mathcal H(\Theta)}\Psi^\top (Sh)_{(n,N)}=
P_{\mathcal H(\Theta)}S_N\Psi^\top h_{(n,N)}=P_{\mathcal H(\Theta)}S_NP_{\mathcal H(\Theta)}\Psi^\top h_{(n,N)} 
\\ &=
T_\Theta P_{\mathcal H(\Theta)}\Psi^\top h_{(n,N)}= T_\Theta Y_n h, \ \ h\in H^2.
\end{align*}
Since $\Phi\Theta_1=\mathbb O$ and $\Theta=\Theta_1\Omega$, we have  
\begin{equation*} \Phi Y_n h=\Phi P_{\mathcal H(\Theta)}\Psi^\top h_{(n,N)}= 
\Phi(\Psi^\top h_{(n,N)}-P_{\Theta H^2_{N-M}}\Psi^\top h_{(n,N)})= \Phi\Psi^\top h_{(n,N)}.
\end{equation*}
It follows from the last equality and \eqref{vartheta} that 
\begin{equation*} \|Y_n h\| = \|\Phi\| \|Y_n h\|\geq  \|\Phi Y_n h\|= \|\Phi\Psi^\top h_{(n,N)}\|\geq \|\vartheta_n h\|=\|h\|, \ \ h\in H^2.
\end{equation*}
(The equality $\|\Phi\|=1$ holds true, because $\Phi$ is $*$-inner.)
Set $\mathcal M_n=Y_n H^2$ and consider $Y_n$ as a transformation from $H^2$ to $\mathcal M_n$. Then 
  $\mathcal M_n\in\operatorname{Lat}T_\Theta$ and $Y_n$ realizes the relation $T_\Theta|_{\mathcal M_n}\approx S$.  

The equality $\mathcal H(\Theta)=\mathcal M_1+\ldots+\mathcal M_N$ 
 follows from the equality $H^2_N=\Psi^\top H^2_N$, which is fulfilled, because $\Psi$ is invertible in 
$H^\infty(\mathbb C^N, \mathbb C^N)$.
\end{proof}

\begin{theorem}\label{thmmain} Let $T\in\mathcal L(\mathcal H)$. Then the following are equivalent:

$\mathrm{(i)}$ there exists a contraction $R$ of class $C_{\cdot 0}$ with $d_R<d_{R^*}<\infty$ such that $T\approx R$; 

$\mathrm{(ii)}$  there exist $N\in\mathbb N$ and  $\mathcal M_n\in\operatorname{Lat}T$, $n=1,\ldots, N$, 
such that $T|_{\mathcal M_n}\approx S$ for every  $n=1,\ldots, N$, and 
$\mathcal H=\mathcal M_1+\ldots+\mathcal M_N$ in the following sense: for every $x\in\mathcal H$ there exist 
  $x_n\in\mathcal M_n$, $n=1,\ldots, N$, such that $x=x_1+\ldots+x_N$. 

Moreover, if $\mathrm{(i)}$ and  $\mathrm{(ii)}$ are fulfilled, then  
\begin{align*}\min\{d_{R^*}\ : \ R \text{ is a contraction such that }T\approx R\}\\=\min\{N\in\mathbb N\ :\ N 
\text{  satisfies assumption }\mathrm{(ii)}\}.
\end{align*} 
\end{theorem}

\begin{proof} The part (i)$\Rightarrow$(ii) is Theorem \ref{thm1}. Indeed, if $R$ is  a contraction of class $C_{\cdot 0}$ 
with $d_R<d_{R^*}<\infty$, then $R\cong T_\Theta$, where $\Theta$ satisfies 
the assumption of Theorem \ref{thm1}, see references in Introduction.  

Assume that (ii) is fulfilled. Let
 $Y_n\in\mathcal I(S,T|_{\mathcal M_n})$ be invertible,   $n=1,\ldots, N$. 
Let $[h_n]_{n=1}^N\in H^2_N$, where $h_n\in H^2$,  $n=1,\ldots, N$. 
Set \begin{equation*}Y[h_n]_{n=1}^N=\sum_{n=1}^NY_n h_n. \end{equation*}
Then \begin{align*}\|Y[h_n]_{n=1}^N\|^2 & \leq\Bigl(\sum_{n=1}^N\|Y_n h_n\|\Bigr)^2
\leq N\Bigl(\sum_{n=1}^N\|Y_n h_n\|^2\Bigr)\\ & \leq N\max_{1\leq n\leq N}\|Y_n\|^2\sum_{n=1}^N\|h_n\|^2
=N\max_{1\leq n\leq N}\|Y_n\|^2\|[h_n]_{n=1}^N\|^2. 
 \end{align*}
Thus, $Y$ is bounded. The equations $YS_N=TY$ and $YH^2_N=\mathcal H$ follow from the definition 
of $Y$ and  the  assumption $\mathcal H=\mathcal M_1+\ldots+\mathcal M_N$. Furthermore, the equation $YS_N=TY$ 
implies $\ker Y\in\operatorname{Lat}S_N$. Set $\mathcal K=H^2_N\ominus\ker Y$, 
$Z=Y|_{\mathcal K}$ and $R=P_{\mathcal K}S_N|_{\mathcal K}$. Then $ZR=TZ$, $Z\mathcal K=YH^2_N=\mathcal H$ 
and $\ker Z=\{0\}$. Consequently, $T\approx R$. Clearly, $R$ is a contraction of class $C_{\cdot 0}$ with 
$d_{R^*}\leq N<\infty$. The relation $T\approx R$ and the assumptions on $T$ imply the existense of 
$\mathcal N\in \operatorname{Lat}R$ such that $R|_{\mathcal N}\approx S$. By \cite{sznf74}, $d_R<d_{R^*}$. 

The equality of minima follows from  the proof of the theorem.
 \end{proof}

\begin{remark} The representation  $x=x_1+\ldots+x_N$, $x_n\in\mathcal M_n$, $n=1,\ldots, N$, 
from the statement (ii) of Theorem \ref{thmmain} is not unique in general. Moreover, 
if such a representation is unique for every $x\in\mathcal H$, then $T\approx S_N$. Indeed, in this case 
the mapping $X\colon \mathcal H\to\oplus_{n=1}^N\mathcal M_n$, $Xx=x_1\oplus\ldots\oplus x_N$, 
where $x_n\in\mathcal M_n$, $n=1,\ldots, N$, are from the representation $x=x_1+\ldots+x_N$, 
is a well-defined linear bijection.
It is easy to see that $X$ and $X^{-1}$ are bounded by the Closed Graph Theorem. 
 Consequently, $T\approx \oplus_{n=1}^N T|_{\mathcal M_n}\approx S_N$, 
where first similarity  is realized by $X$, and second similarity follows from the assumption 
$T|_{\mathcal M_n}\approx S$ for every  $n=1,\ldots, N$. 
\end{remark}

\begin{remark} 
In \cite{vm}, for every $N\in\mathbb N$,  $N\geq 2$,  an example of a contraction $T$ is constructed such that 
$d_T+1=d_{T^*}=N$ and $T\prec S$, but $T$ is not quasisimilar (and, consequently, not similar) to any contraction 
$R$ with $d_{R^*}<N$.  The relation $T\prec S$ implies  $\mu_T\leq 2$, see \cite{sznf74}. 
(For further results on multiplicity of contractions with finite defects (not necessarily of class $C_{\cdot 0}$) see \cite{vasmult}.)
Let  $N\in\mathbb N$,  $N\geq 2$, and let $T$ be a contraction from \cite{vm} with $d_{T^*}=N$ described just above. 
Denote by $\mathcal H$ the space on which $T$ acts. The minimal quantity of shift-type invariant subspaces of $T$ 
whose \emph{linear} span is  $\mathcal H$ is equal to $N$ by \cite{vm} and Theorem \ref{thmmain},  while  the minimal quantity of shift-type invariant subspaces of $T$ 
whose \emph{closed linear} span is  $\mathcal H$ is equal to $2$ by \cite{gamal22} and \cite{sznf74}. (For a generalization of \cite{vm} see \cite{popescu}.) 
\end{remark}

\section{On similarity to contractions of class $C_0(N)$}

Let $N\in\mathbb N$. 
Let $T$ be a contraction of class  $C_{\cdot 0}$ with $d_T=d_{T^*}=N<\infty$. 
Equivalently, let $T$ be of class $C_0(N)$ (see \cite{sznf74}). 
Since $T$ is of class $C_0$, 
there is no shift-type invariant subspace of $T$, that is, such invariant subspace that the restriction of  $T$ on 
this subspace is similar to $S$. 
Moreover, it is not possible to obtain a result similar to Theorem \ref{thm1}  even if $S$ will be replaced by 
contractions $T_{\vartheta_n}$ for inner functions $\vartheta_n\in H^\infty$,   $n=1,\ldots, N$, different from each other, 
because contractions of class $C_0$ can be unicellular (the definition is recalled in Introduction). 
 If $T$ is unicellular and 
 $\mathcal H=\mathcal M_1+\ldots+\mathcal M_M$ for some $\mathcal M_n\in\operatorname{Lat}T$, $n=1,\ldots, M$, 
 $M\in\mathbb N$,
then $\mathcal H=\mathcal M_n$ for some $n$, $1\leq n\leq M$. Consequently, a statement  analogous  to Theorem \ref{thm1} must imply that 
every unicellular contraction of class $C_0(N)$ must be similar to $T_\vartheta$ for some inner function $\vartheta\in H^\infty$. 
In Example \ref{exaunicellular} below a unicellular contraction of class $C_0(2)$ is constructed which is not similar to 
 $T_\vartheta$ for any inner function $\vartheta\in H^\infty$. 

On the other side, the following proposition analogous to the statement (ii)$\Rightarrow$(i) in Theorem \ref{thmmain} holds true. 

\begin{proposition}\label{propc0} Let  $N\in\mathbb N$, and let $T\in\mathcal L(\mathcal H)$.
 Suppose that there exist  inner functions $\vartheta_n\in H^\infty$ and $\mathcal M_n\in\operatorname{Lat}T$, $n=1,\ldots, N$, 
such that $T|_{\mathcal M_n}\approx T_{\vartheta_n}$ for every  $n=1,\ldots, N$, and 
$\mathcal H=\mathcal M_1+\ldots+\mathcal M_N$ in the following sence: for every $x\in\mathcal H$ there exist 
  $x_n\in\mathcal M_n$, $n=1,\ldots, N$, such that $x=x_1+\ldots+x_N$.

Then there exists a contraction $R$ of class $C_0(M)$ with $M\leq N$ such that $T\approx R$. 
\end{proposition}

\begin{proof} 
 Let 
$Y_n\in\mathcal I(T_{\vartheta_n},T|_{\mathcal M_n})$  be invertible,  $n=1,\ldots, N$. 
Let $h_n\in\mathcal H(\vartheta_n)$,  $n=1,\ldots, N$. 
Set \begin{equation*}Y(\oplus_{n=1}^N h_n)=\sum_{n=1}^NY_n h_n. \end{equation*}
The proof of the boundedness of $Y$ is the same as in the proof of Theorem \ref{thmmain}. 
Set \begin{equation*} \Theta=\left[\begin{matrix}\vartheta_1 & 0 & \ldots & 0 \\
0 & \vartheta_2 & \ldots & 0 \\ 
 \vdots & \vdots & \ddots & \vdots \\
0 & 0 & \ldots & \vartheta_N \end{matrix}\right].\end{equation*} 
Then $\mathcal H(\Theta)=\oplus_{n=1}^N  \mathcal H(\vartheta_n)$ and $T_\Theta=\oplus_{n=1}^N  T_{\vartheta_n}$. 
 The equations $Y  T_\Theta=TY$ and $Y \mathcal H(\Theta)=\mathcal H$ follow from the definition 
of $Y$ and the  assumption $\mathcal H=\mathcal M_1+\ldots+\mathcal M_N$. Furthermore, 
the equation $YT_\Theta=TY$ 
implies $\ker Y\in\operatorname{Lat}T_\Theta$. Set $\mathcal K= \mathcal H(\Theta)\ominus\ker Y$, 
$Z=Y|_{\mathcal K}$ and $R=P_{\mathcal K}T_\Theta|_{\mathcal K}$. Then $R$ is a contraction of class $C_0(M)$ 
with $M\leq N$, because $T_\Theta$ is a contraction of class $C_0(N)$. Furthermore, 
 $ZR=TZ$, $Z\mathcal K=Y \mathcal H(\Theta)=\mathcal H$ 
and $\ker Z=\{0\}$. Consequently, $T\approx R$. 
\end{proof}

The following lemma is a corollary of results from \cite{sznfbk} and \cite{berc}. This lemma is used in Example  \ref{exaunicellular}. 

\begin{lemma} \label{lemexa} Let $N\in\mathbb N$, and let $\Theta\in H^\infty(\mathbb C^N,\mathbb C^N)$ be inner. 
Denote by $\Theta^{Ad}$ the algebraic adjoint of $\Theta$. Let $\Theta^{Ad}$ be 
represented by a matrix $\Theta^{Ad}=[\theta_{nk}]_{n,k=1}^N$ with respect to some orthonormal basis of 
 $\mathbb C^N$. 
Furthermore, let $\vartheta\in H^\infty$ be inner. Then $T_\Theta\sim T_\vartheta$ if and only if $\det\Theta=\vartheta$ and 
the greatest common inner divisor of  $\theta_{nk}$, $n,k=1,\ldots,N$, is $1$. If $T_\Theta\approx T_\vartheta$, then  
\begin{equation*} \inf_{z\in\mathbb D}\sum_{n,k=1}^N|\theta_{nk}(z)|>0.  
\end{equation*}
\end{lemma}
 
\begin{proof} The statement on the quasisimilarity $T_\Theta\sim T_\vartheta$ is contained in {\cite[Theorem X.6.6]{sznfbk}} (see also {\cite[Proposition VI.5.11]{berc}}), 
so its proof is omitted. Let $T_\Theta\approx T_\vartheta$. Let $X\in\mathcal I(T_\vartheta,T_{\Theta})$ 
be  invertible. By {\cite[Theorem VI.3.6]{sznfbk}}, there exist
$\Psi_1$, $\Psi_3\in  H^\infty(\mathbb C,\mathbb C^N)$ and 
$\Psi_2$, $\Psi_4\in  H^\infty(\mathbb C^N,\mathbb C)$ 
such that $Xf=P_{\mathcal H(\Theta)}\Psi_1 f$
 for every $f \in\mathcal H(\vartheta)$, 
 $X^{-1}h=P_{\mathcal H(\vartheta)}\Psi_4 h$
 for every $h \in\mathcal H(\Theta)$, $\Psi_1\vartheta=\Theta\Psi_3$, and $\Psi_4\Theta=\vartheta\Psi_2$. 
Since $\Theta^{Ad}\Theta = \Theta\Theta^{Ad}=(\det\Theta)\cdot I_{\mathbb C^N}=\vartheta I_{\mathbb C^N}$, we obtain
  $\Psi_4=\Psi_2\Theta^{Ad}$. For every $f\in\mathcal H(\vartheta)$ we have  
\begin{align*} f&=X^{-1}Xf=P_{\mathcal H(\vartheta)}\Psi_4 P_{\mathcal H(\Theta)}\Psi_1 f=
P_{\mathcal H(\vartheta)}\Psi_2\Theta^{Ad} P_{\mathcal H(\Theta)}\Psi_1 f \\ &= 
P_{\mathcal H(\vartheta)}\Psi_2\Theta^{Ad}(I_{H^2_N}- P_{\Theta H^2_N})\Psi_1 f= 
P_{\mathcal H(\vartheta)}\Psi_2\Theta^{Ad}\Psi_1 f.
\end{align*}
Therefore, there exists $\psi_0\in H^\infty$ such that $\Psi_2\Theta^{Ad}\Psi_1=1+\psi_0\vartheta$. Denote by 
$\psi_{1n}$ and $\psi_{2n}$, $n=1,\ldots, N$, the elements of the column $\Psi_1$ and the row $\Psi_2$, respectively. 
Then  $\Psi_2\Theta^{Ad}\Psi_1=\sum_{n,k=1}^N\theta_{nk}\psi_{1k}\psi_{2n}=  1+\psi_0\vartheta$. Consequently, 
 \begin{equation*} \inf_{z\in\mathbb D}\Bigl(\sum_{n,k=1}^N|\theta_{nk}(z)|+|\vartheta(z)|\Bigr)>0.  
\end{equation*}
Since $\det\Theta^{Ad}=\vartheta^{N-1}$, we have 
 \begin{equation*} |\vartheta(z)|^{N-1}=|\det\Theta^{Ad}(z)|\leq N!\Bigl( \max_{n,k=1,\ldots, N}|\theta_{nk}(z)|\Bigr)^N, \ \ z\in\mathbb D. 
\end{equation*}
Consequently, if $ \inf_{z\in\mathbb D}\sum_{n,k=1}^N|\theta_{nk}(z)|=0$, then 
$\inf_{z\in\mathbb D}(\sum_{n,k=1}^N|\theta_{nk}(z)|+|\vartheta(z)|)=0$, a contradiction. 
\end{proof} 

\begin{example} \label{exaunicellular} Let $\vartheta_1$, $\vartheta_2\in H^\infty$ be two inner functions, and let $c\in\mathbb C$ be such that 
$|c|^2=\frac{1}{1+|\vartheta_1(0)|^2}$. Set  $\chi(z)=z$, $z\in\mathbb D$, and 
\begin{equation}\label{exaexp}
\Theta=\left[\begin{matrix} \frac{c\vartheta_2}{2}\frac{(1-\chi)\vartheta_1-(1+\chi)\vartheta_1(0)}{\chi} & 
\frac{c}{2}\frac{(1+\chi)\vartheta_1-(1-\chi)\vartheta_1(0)}{\chi} \\ \empty &\empty \\
\frac{\overline c\vartheta_2}{2}(1+\chi +(1-\chi)\vartheta_1\overline{\vartheta_1(0)}) & 
\frac{\overline c}{2}(1-\chi+(1+\chi)\vartheta_1\overline{\vartheta_1(0)})\end{matrix}\right].
\end{equation}
Clearly,  $\Theta\in H^\infty(\mathbb C^2,\mathbb C^2)$. 
Straightforward calculation shows that $\Theta^*(\zeta)\Theta(\zeta)=I_{\mathbb C^2}$ for a.e. $\zeta\in \mathbb T$. Consequently, 
$\Theta$ is an inner function (see the references in Introduction). We have $\det \Theta=-\vartheta_1\vartheta_2$. 
Denote the elements of the matrix $\Theta^{Ad}$ by $\theta_{nk}$, $n,k=1,2$. 
Then  $\Theta=\left[\begin{matrix}\theta_{22} & -\theta_{12}\\ -\theta_{21} & \theta_{11}\end{matrix}\right]$ and 
$\Theta^{Ad}=\left[\begin{matrix}\theta_{11} & \theta_{12}\\ \theta_{21} & \theta_{22}\end{matrix}\right]$. 
If there exists a sequence $\{z_l\}_l\subset\mathbb D$
 such that $z_l\to 1$ and  $\vartheta_k(z_l)\to 0$ for $k=1,2$, then 
\begin{equation*} \inf_{z\in\mathbb D}\sum_{n,k=1}^2|\theta_{nk}(z)|=0.  
\end{equation*}
For $a>0$ set 
\begin{equation}\label{defexp}\alpha_a(z)=\exp\Bigl(a\frac{z+1}{z-1}\Bigr), \ \  z\in\mathbb D. 
\end{equation} 
Now let $a_k>0$,  and let $\vartheta_k(z)=\alpha_{a_k}$, $k=1,2$. Then $\vartheta_k(z)\to 0$ for $k=1,2$,
when $z\in(0,1)$, $z\to 1$. 
Furthermore, 
the greatest common inner divisor of  $\theta_{nk}$, $n,k=1,2$, is $1$. Indeed, if $\alpha$ is  
the greatest common inner divisor of  $\theta_{nk}$, $n,k=1,2$, then $\alpha$ is a divisor of 
$\det\Theta=-\alpha_{a_1+a_2}$. Consequently,  $\alpha=\alpha_{a_0}$ for some $a_0>0$. 
In particular, $\theta_{11}=\alpha_{a_0}g_0$ for some $g_0\in H^\infty$. 
By definition, \begin{equation*} \theta_{11} =
\frac{\overline c}{2}(1-\chi+(1+\chi)\vartheta_1\overline{\vartheta_1(0)})=
\frac{\overline c}{2}(1-\chi+(1+\chi)\alpha_{a_1}\overline{\alpha_{a_1}(0)}).
\end{equation*}
It follows   $1-\chi=\alpha_{\min(a_1,a_0)}g$ for some $g\in H^\infty$, a contradiction. 

Thus, if $\Theta$ is defined by \eqref{exaexp} with 
 $\vartheta_k(z)=\alpha_{a_k}$ for some $a_k>0$, $k=1,2$, where  $\alpha_{a_k}$ are defined by \eqref{defexp}, 
then Lemma  \ref{lemexa} and the consideration above imply that $T_\Theta\sim T_{\alpha_{a_1+a_2}}$ and 
$T_\Theta\not\approx T_{\alpha_{a_1+a_2}}$. By {\cite[Exersices IV.1.13 and IV.1.14]{berc}}, $T_\Theta$ is unicellular (because $T_\Theta\sim T_{\alpha_a}$ for some $a>0$; see also {\cite[Sec. III.7.2]{sznfbk}}, {\cite[Sec. IV.3]{nik86}} and 
{\cite[Proposition VII.1.21]{berc}}). 
\end{example}

\section{On operators of class $C_0$ whose minimal function is a Carleson--Newman product}

Remember the definitions. 
For $\lambda\in\mathbb D$ set 
$b_\lambda(z)= \frac{|\lambda|}{\lambda}\frac{\lambda-z}{1-\overline\lambda z}$, $z\in\mathbb D$, 
if $\lambda\neq 0$,  and $b_\lambda(z)=z$, $z\in\mathbb D$, if $\lambda=0$.
It is well known that if $\{\lambda_n\}_n\subset\Bbb D$, then  the product $B=\prod_nb_{\lambda_n}$  converges if and only if 
 $\sum_n(1-|\lambda_n|)<\infty$ (that is, $\{\lambda_n\}_n$ satisfies the \emph{Blaschke condition}). The product 
 $B$ is called a \emph{Blaschke product}. 
Let a sequence $\{\lambda_n\}_n\subset\Bbb D$ be such that  $\lambda_n\neq\lambda_k$, if $n\neq k$, 
and  $\sum_n(1-|\lambda_n|)<\infty$.  
  Set $B_n=\prod_{k\neq n}b_{\lambda_k}$. The sequence $\{\lambda_n\}_n$ 
satisfies the \emph{Carleson interpolating condition} (the Carleson condition for brevity), if $\inf_n|B_n(\lambda_n)|>0$. 
A \emph{Carleson--Newman product} is a finite 
product of Blaschke products with simple zeros satisfying the Carleson condition.

\begin{lemma}\label{lemblaschke} Let  $\Lambda\subset\mathbb D$ satisfy  
the  Blaschke  condition and the Carleson condition. 
  Set $B=\prod_{\lambda\in\Lambda}b_\lambda$. Let $T$ be a polynomially bounded operator of class $C_0$. Suppose that 
$B$ is the minimal  function of $T$ and $M=\mu_T<\infty$. 
Set
\begin{align*} \Lambda_n& =\{\lambda\in\Lambda \ :\ \dim\ker (T-\lambda I)\geq n\}, 
\ \   B_n=\prod_{\lambda\in\Lambda_n}b_\lambda, \ \ n\in\mathbb N, \\ \text{ and } 
\Theta&=\left[ \begin{matrix} B_1  & 0 & \ldots & 0 \\ 
0 & B_2 & \ldots & 0\\ 
\vdots & \vdots &\ddots &\vdots \\
 0 & 0 & \ldots & B_M
 \end{matrix}\right]. 
\end{align*} 
Then $T\approx T_\Theta$. 
\end{lemma}

\begin{proof} Set $k(\lambda)=\dim\ker (T-\lambda I)$, $\lambda\in\Lambda$, and  
 \begin{equation*} R=\oplus_{\lambda\in\Lambda}\lambda I_{\mathbb C^{k(\lambda)}}.
\end{equation*} 
It is well known (see, for example, {\cite[proof of Theorem 2.1]{cassierest}}, {\cite[Corollary 3.3]{clouatre11}}, 
{\cite[Lemma 2.3]{vitse}})
that $T\approx R$. Since $\ker (R-\lambda I)$ is a reducing subspace for $R$ and $\mu_{R|_{\ker (R-\lambda I)}}=k(\lambda)$, 
we have $k(\lambda)\leq \mu_R=\mu_T=M$. 

Since $T_{B_n}\approx\oplus_{\lambda\in\Lambda_n}\lambda I_{\mathbb C}$ and 
$T_\Theta\cong \oplus_{n=1}^M T_{B_n}$, we have  
\begin{equation*} T_\Theta\approx \oplus_{n=1}^M\oplus_{\lambda\in\Lambda_n}\lambda I_{\mathbb C}=R.
\end{equation*} 
Thus, $T\approx T_\Theta$.  
\end{proof}

In the next two lemmas, one deals with  a triangulation of an operator $T$ of the form \eqref{triang} just below, 
 which is given before the formulations of the lemmas to avoid repetition: 
\begin{equation}\label{triang} T=\left[ \begin{matrix} T_1  & * & \ldots & * & * \\ 
\mathbb O & T_2 & \ldots & * & * \\ \vdots & \vdots & \ddots & \vdots & \vdots \\
\mathbb O & \mathbb O & \ldots & T_{N-1} & * \\ \mathbb O & \mathbb O & \ldots & \mathbb O & T_N
 \end{matrix}\right]. 
\end{equation}

\begin{lemma} \label{lemthetacc0}   Let $N\in\mathbb N$,  and let $\vartheta_n$, $n=1,\ldots, N$, be inner functions. 
 Set $\vartheta=\prod_{n=1}^N \vartheta_n$. Suppose that  $T$ is an a.c. continuous polynomially bounded operator 
and $\vartheta(T)=\mathbb O$. Then $T$ has a triangulation of the form \eqref{triang},  
where $\vartheta_n(T_n)=\mathbb O$,  $n=1,\ldots, N$. 
\end{lemma}

\begin{proof} The lemma will be proved using induction. Let $N=2$. Denote by $\mathcal H$ the space on which $T$ acts. 
Set $\mathcal H_1=\ker\vartheta_1(T)$, $T_1=T|_{\mathcal H_1}$,
 $\mathcal H_2=\mathcal H\ominus\mathcal H_1$, and $T_2=P_{\mathcal H_2}T|_{\mathcal H_2}$.  
Since 
\begin{equation*}\mathcal H_2=\operatorname{clos}(\vartheta_1(T))^*\mathcal H
=\operatorname{clos}\widetilde\vartheta_1(T^*)\mathcal H
\end{equation*} 
and 
\begin{equation*} \{0\}=\operatorname{clos}\widetilde\vartheta_2(T^*)\widetilde\vartheta_1(T^*)\mathcal H =
\operatorname{clos}\widetilde\vartheta_2(T^*)\mathcal H_2,
\end{equation*} 
we conclude that $\widetilde\vartheta_2(T^*)|_{\mathcal H_2}=\mathbb O$. Consequently,  
$\mathbb O= (\widetilde\vartheta_2(T^*)|_{\mathcal H_2})^*=\vartheta_2(T_2)$. 

If $N>2$, assume that the lemma is proved for $N-1$. Apply the case $N=2$ to  $T$ and  the representation  
 $\vartheta=(\prod_{n=1}^{N-1}\vartheta_n)\cdot\vartheta_N$. 
Then \begin{equation*} T= \left[ \begin{matrix} R  & * \\ \mathbb O & T_N
 \end{matrix}\right], 
\end{equation*} 
where $(\prod_{n=1}^{N-1}\vartheta_n)(R)=\mathbb O$ and $\vartheta_N(T_N)=\mathbb O$. 
The conclusion of the lemma follows from the obtained representation of $T$ and  the inductive hypothesis applied to $R$. 
\end{proof}

The proof of the following lemma is based on {\cite[Lemma 1]{takshift}}.

\begin{lemma} \label{lemisometry}  Suppose that $N\in\mathbb N$, 
  $T\in\mathcal L(\mathcal H)$ is a contraction, and $T$ has a triangulation of the form \eqref{triang}. 
Denote by $\mathcal H_n$ the space on which $T_n$ acts, $n=1,\ldots, N$. 
Suppose that for every $n$, $n=1,\ldots, N$, there exist $M_n\in\mathbb N$ and  
  $Y_n\in\mathcal I(S_{M_n},T_n)$ such that $Y_n H^2_{M_n}=\mathcal H_n$. Set $M=\sum_{n=1}^N M_n$. 
Then there exists 
$Y\in\mathcal I(S_M,T)$ such that $Y H^2_M=\mathcal H$. 
\end{lemma}

\begin{proof} The lemma will be proved using induction. Let $N=2$. Then $T$  has the form 
$T= \left[ \begin{matrix} T_1  & A \\ \mathbb O & T_2 \end{matrix}\right]$ for some transformation $A$.
  By   {\cite[Lemma 1]{takshift}} applied to $T_2$ and $Y_2$, 
there exists $Z\in\mathcal L(H^2_{M_2},\mathcal H_1)$ such that 
 \begin{equation*}  
\left[ \begin{matrix} I_{\mathcal H_1}  & Z \\ \mathbb O & Y_2 \end{matrix}\right] 
\left[ \begin{matrix} T_1  & \mathbb O \\ \mathbb O & S_{M_2} \end{matrix}\right] 
= \left[ \begin{matrix} T_1  & A \\ \mathbb O & T_2 \end{matrix}\right]
\left[ \begin{matrix} I_{\mathcal H_1}   & Z \\ \mathbb O & Y_2 \end{matrix}\right]. 
\end{equation*} 
Set $Y=\left[ \begin{matrix} Y_1  & Z \\ \mathbb O & Y_2 \end{matrix}\right]$. It is easy to see that $Y$ 
satisfies the conclusion of the lemma for $N=2$. 

If $N>2$, assume that the lemma is proved for $N-1$. Write  
\begin{equation*} T= \left[ \begin{matrix} R  & * \\ \mathbb O & T_N
 \end{matrix}\right], 
\end{equation*} 
where $R$ has the form \eqref{triang} for $N-1$. Let $W$ be a transformation obtained 
by applying the inductive hypothesis to $R$. 
Apply the case $N=2$ to the operator $T$ and transformations $W$ and $Y_N$. 
The obtained transformation satisfies the conclusion of the lemma. 
 \end{proof}

\begin{theorem}\label{thmblaschke} Let $N\in\mathbb N$, and let $B_n$, $n=1,\ldots, N$, be  Blaschke products with simple zeros satisfying 
Carleson condition. Set $B=\prod_{n=1}^N B_n$. Let $T$ be an a.c. continuous polynomially bounded operator.
Suppose that $B(T)=\mathbb O$ and $\mu_T<\infty$. Then $T$ is similar to a contraction with finite defects. 
\end{theorem}

\begin{proof} It is well known (see, for example, {\cite[Theorem 4.4]{clouatre14}} or {\cite[Theorem 3.2]{gamal17}}, 
based on  
\cite{vas}, see also {\cite[Ch. IX]{nik86}}, {\cite[Sec. C.3]{nik02}})
that if $T$ is an a.c. continuous polynomially bounded operator,  
$B$ satisfies   the assumption of the theorem,  and 
$B(T)=\mathbb O$, then $T$ is similar to a contraction. Therefore, 
 we may assume that $T$ is a contraction.  Denote by $\mathcal H$ the space on which $T$ acts. 
Apply Lemma \ref{lemthetacc0}  to $T$ and $B=\prod_{n=1}^N B_n$. We obtain 
   a triangulation of $T$ of the form \eqref{triang},  
where $B_n(T_n)=\mathbb O$,  $n=1,\ldots, N$. By {\cite[Corollaries III.5.26 and III.5.27]{berc}},
$M_n:=\mu_{T_n}\leq\mu_T$ for every $n=1,\ldots, N$. 
Let $B_{1n}$ be the minimal function of $T_n$. Since $B_{1n}$ divides $B_n$, the sequence of zeros of $B_{1n}$ satisfies  the Carleson condition. 
Let $\Theta_n\in H^\infty(\mathbb C^{M_n},\mathbb C^{M_n})$ be a  function from  Lemma \ref{lemblaschke} applied to $T_n$, $n=1,\ldots, N$. 
Then $T_n\approx T_{\Theta_n}$. Denote by $\mathcal H_n$ the space on which $T_n$ acts. 
 There exists an invertible transformation $X_n\in\mathcal I(T_{\Theta_n},T_n)$.
 Set $Y_n=P_{\mathcal H(\Theta_n)}X_n$. 
Then $Y_n\in\mathcal I(S_{M_n},T_n)$ 
 and  $Y_n H^2_{M_n}=\mathcal H_n$. 

Set $M=\sum_{n=1}^N M_n$. Let  $Y$ 
 be a transformation 
 from the conclusion of Lemma \ref{lemisometry} applied to $T$. Then 
 $Y S_M=T Y$ and $Y H^2_M=\mathcal H$. 
Set 
  $\mathcal K=H^2_M\ominus\ker Y$ and $R=P_{\mathcal K} S_M|_{\mathcal K}$. 
Then $R$ is a contraction with finite defects and $Y|_{\mathcal K}$ realizes the similarity $R\approx T$. 
\end{proof}

\end{document}